\RequirePackage{fix-cm}
\documentclass[smallextended]{svjour3}       
\smartqed  
\usepackage{graphicx}
\usepackage{amssymb}
\usepackage{amsmath}
\usepackage{tikz}
\usepackage{textcomp}
\begin{document}

\title{Cauchy formulae and Hardy spaces in discrete octonionic analysis 
}


\author{Rolf S\"oren Krau\ss har \and Dmitrii Legatiuk}

\authorrunning{R.S. Krau\ss har, D. Legatiuk} 

\institute{Rolf S\"oren Krau\ss har \at
	       University of Erfurt \\
		   Chair of Mathematics \\
           Nordh\"auser Str. 63, 99089 Erfurt, Germany \\
           \email{soeren.krausshar@uni-erfurt.de}    
           \and
		   Dmitrii Legatiuk \at
           University of Erfurt \\
		   Chair of Mathematics \\
           Nordh\"auser Str. 63, 99089 Erfurt, Germany \\
           \email{dmitrii.legatiuk@uni-erfurt.de}
}

\date{Received: date / Accepted: date}
\dedication{This paper is dedicated to John Ryan on the occasion of his retirement}

\maketitle

\begin{abstract}
In this paper, we continue the development of a fundament of discrete octonionic analysis that is associated to the discrete first order Cauchy-Riemann operator acting on octonions. In particular, we establish a discrete octonionic version of the Borel-Pompeiu formula and of Cauchy's integral formula. The latter then is exploited to introduce a discrete monogenic octonionic Cauchy transform. This tool in hand allows us to introduce discrete octonionic Hardy spaces for upper and lower half-space together with Plemelj projection formulae.

\keywords{Octonions \and Discrete Dirac operator \and Discrete monogenic functions \and Discrete octonionic function theory \and Discrete Cauchy transform \and Hardy spaces}
\subclass{39A12 \and 42A38 \and 44A15}
\end{abstract}

\section{Introduction}

Complex analytic tools together with their related discretisation methods provide a very powerful toolkit to compute numerically and some cases even analytically the solutions to many partial differential equations with given boundary data that arise in the context of harmonic analysis. In the focus there are null-solutions to a discrete version of the Cauchy-Riemann operator that satisfy a discretised version of the Cauchy integral and the Borel-Pompeiu formula. These integral formulas provide us with the key ingredients in the representation of the solutions.\par
Now, there are several possibilities to generalise complex function theory together with discretisations to higher dimensional settings. Many branches of engineering particularly focus on three-dimensional generalisations. To address the three-dimensional setting the function theory related to associative Clifford algebras has proven to be a very effective tool. In this context one considers a three-dimensional Cauchy-Riemann operator defined by $D := \sum\limits_{i=0}^2 \partial_{x_i} \mathbf{e}_i$ where $\mathbf{e}_1$ and $\mathbf{e}_2$ are two different imaginary units generating the Clifford algebra $\mathcal{C}\ell_{0,2}$ which is spanned as an $\mathbb{R}$-vector space by the basis elements $1,\mathbf{e}_1,\mathbf{e}_2,\mathbf{e}_1 \mathbf{e}_2$. Actually the particular case $\mathcal{C}\ell_{0,2}$ represents the Hamiltonian quaternions which for a non-commutative skew-field.\par
Using the multiplicative structure of the Clifford algebra, the generalised first order Cauchy-Riemann operator factorises the second-order Laplacian in $\mathbb{R}^3$. In fact, this approach could be generalised very easily to all associative Clifford algebras. Its function theory now is broadly known under the term {\itshape Clifford analysis} while functions in the kernel of the generalised Cauchy-Riemann operator are often called monogenic, hyperholomorphic  or Clifford holomorphic functions. Classical references are for instance the books \cite{BDS,GS2} as well as also J. Ryan's edited volume containing the milestone contribution \cite{McIntosh} among others which all provided a boost in the development of this function theory including applications to Calderon-Zygmund type operators. Particularly, \cite{GS2} exhibits how related integral operators, in particular Plemelj-Sokhotzkij type formulas related to the Cauchy transform, can successfully be applied to treat boundary value problems. Additionally, J. Ryan and his co-authors have also addressed unbounded domains with this function theoretical toolkit in \cite{GKR}. \par
Over the last two decades there has been a growing interest in the discretisation of the continuous Clifford analysis for developing numerical algorithms for the higher dimensional associative framework, cf. for instance  \cite{Brackx_1,Cerejeiras_1,CKKS,Cerejeiras_2,Cerejeiras,Faustino_1,FKS07,Guerlebeck_1} among others. However, instead of embedding the vector space $\mathbb{R}^{n+1}$ into associative Clifford algebras $\mathcal{C}\ell_{0,n}$, which have zero-divisors for $n>2$, there are many more ways to generalise complex analysis to higher dimensions. If the Cayley-Dickson duplication process to the complex numbers is applied, then we first arrive at the four-dimensional Hamiltonian quaternions, which, as mentioned is still a particular case of a Clifford algebra. However, if we apply, as the next step, the Cayley-Dickson doubling to the quaternions, then we obtain a different sort of  algebra, namely the {\itshape octonions} or {\itshape Cayley numbers}, denoted by $\mathbb{O}$ in all that follows, cf. for instance \cite{Baez} among others. The octonions turn out to be non-associative. Consequently they do not belong to the Clifford algebras and they are not representable with matrices. But, in contrast to the eight-dimensional Clifford algebra $\mathcal{C}\ell_{0,3}$ octonions still form a normed division algebra (in the wider sense of being non-associative without zero-divisors). So, every non-zero element has a multiplicative inverse. According to the famous theorem of Hurwitz, they form the largest real normed zero divisor-free algebra over $\mathbb{R}$.\par
Surprisingly, although the octonions are not associative, according to recently published research results in particle physics, see for instance \cite{Burdik,G,NSM}, they seem to offer a much more adequate model for a unified description of particle physics including gravity, which was also already proposed in the book by F. G\"ursey and H. Tze in 1996, see \cite{GTBook}. This provides one further motivation to develop analytic and discretised function theoretic tools in the non-associative context of octonions. On the continuous level at least the fundaments of an octonionic function theory are already well-developed, see for example \cite{DS,Nono,XL2000,XL2002}. In \cite{XLT2008} generalisations of the Cauchy transform together with Plemelj projection formulas and with some basic applications to Calderon-Zygmund type operators have been presented particularly. Recently one also managed to introduce meaningful octonionic generalisations of Bergman and Hardy spaces on the continuous level, see for example  \cite{WL2018,WL2020,ConKra2021,CKS2023}, as well as the papers \cite{FL,QR2021} addressing octonionic Hilbert spaces on a more general level.\par
Nevertheless, if we want to solve practical octonionic boundary value problems numerically, then we are in need of discretised versions of these octonionic operators. As far as we know, the development of a discrete octonionic function theory is still a rather open research field. In our recent paper \cite{Krausshar_1}, we developed a fundament for research in this direction, namely, we introduced a discretised version of octonionic Cauchy-Riemann operator in terms of appropriate forward and backward operators in the Hermitian sense, and established a discrete Stokes' formula for that operator.\par
In this paper, we depart from the discretised Stokes' formula and exploit that one further in order to obtain a discrete octonionic version of Borel-Pompeiu's formula. The latter in turn produces a discrete Cauchy formula in a special case. As a next step, we establish a discretised version of the octonionic Cauchy transform that has been introduced on the continuous level in \cite{XLT2008}. The related Plemelj-Sokhotzkij projection formulae then give rise to the definition of a discrete octonionic monogenic generalisation of Hardy spaces. In particular, the focus of this paper is put on the consideration of the upper octonionic half-space (resp. its related upper half-lattice) where the last component satisfies $x_7>0$. In this case, we have a simple geometric description of the inner product that has to be used here. An alternative way to introduce discrete octonionic monogenic Hardy spaces consists in defining appropriate extension operators using the Fourier transform on the complexified octonions and its related Fourier symbols, which in our framework are also complexified octonions. Also this approach will be carefully developed  in this paper as it has the advantage of being fully explicit. \par
As a consequence of the lack of associativity, the signs in the discrete octonionic Borel-Pompeiu and Cauchy differ from the results in discrete Clifford analysis. Surprisingly, the discretisation of the continuous octonionic analysis that we develop here has the interesting effect that the associator expressions that we intrinsically have in the Stokes formula of the continuous case (see \cite{XLT2008}), disappear in our discretised constructions. This effect also shows up when considering instead of the forward and backward operator its centralised version to which we also pay a particular attention in this paper to get a more complete view on this particular phenomenon.\par
The discrete octonionic setting thus turns out to be really different as well from its associative version in Clifford analysis but also from continuous octonionic monogenic function theory. It exposes new intrinsic peculiarities that have to be taken really carefully into account.\par

\section{Preliminaries and notations}\label{Pre}

\subsection{Continuous octonionic analysis}

In this section, we briefly recall basic notions and the most fundamental results on continuous octonionic analysis. The basic context is the $8$-dimensional Euclidean vector space $\mathbb{R}^{8}$, where the standard vectors are denoted by $\mathbf{e}_{k}$, $k=0,1,\ldots,7$. A vector from $\mathbb{R}^8$ can be expressed as usual in terms of its real coordinates in the way  $\mathbf{x}=(x_{0}, x_{1},\ldots, x_{7})$. Next, $\mathbb{R}^8$-vectors can also be described as octonions 
\begin{equation*}
x = x_{0}\mathbf{e}_{0}+x_{1}\mathbf{e}_{1}+x_{2}\mathbf{e}_{2}+x_{3}\mathbf{e}_{3}+x_{4}\mathbf{e}_{4}+x_{5}\mathbf{e}_{5}+x_{6}\mathbf{e}_{6}+x_{7}\mathbf{e}_{7},
\end{equation*}
where we now additionally identify $\mathbf{e}_{4}=\mathbf{e}_{1}\mathbf{e}_{2}$, $\mathbf{e}_{5}=\mathbf{e}_{1}\mathbf{e}_{3}$, $\mathbf{e}_{6}=\mathbf{e}_{2}\mathbf{e}_{3}$ and $\mathbf{e}_{7}=\mathbf{e}_{4}\mathbf{e}_{3}=(\mathbf{e}_{1}\mathbf{e}_{2})\mathbf{e}_{3}$. Moreover, we have $\mathbf{e}_{i}^{2}=-1$ and  $\mathbf{e}_{0}\mathbf{e}_{i}=\mathbf{e}_{i}\mathbf{e}_{0}$ for all $i=1,\ldots,7$, and $\mathbf{e}_{i}\mathbf{e}_{j}=-\mathbf{e}_{j}\mathbf{e}_{i}$ for all mutual distinct $i,j\in\left\{1,\ldots,7\right\}$, as well as $\mathbf{e}_{0}$ is the neutral element and, therefore, often will be omitted. This definition endows $\mathbb{R}^8$ additionally with a multiplicative closed structure. \par
Table~\ref{Table_octonions} fully describes the multiplication rules for real octonions. This table clearly indicates in particular that the octonionic multiplication actually is closed but not associative, for instance we have $(\mathbf{e}_{i}\mathbf{e}_{j})\mathbf{e}_{k}=-\mathbf{e}_{i}(\mathbf{e}_{j}\mathbf{e}_{k})$.\par
\begin{table}[h]
\caption{Multiplication table for real octonions $\mathbb{O}$}
\label{Table_octonions}
\begin{center}
\begin{tabular}{|c|cccccccc|}
\hline
$\cdot$ & $\mathbf{e}_{0}$ & $\mathbf{e}_{1}$ & $\mathbf{e}_{2}$ & $\mathbf{e}_{3}$ & $\mathbf{e}_{4}$ & $\mathbf{e}_{5}$ & $\mathbf{e}_{6}$ & $\mathbf{e}_{7}$ \\[2pt]
\hline
$\mathbf{e}_{0}$ & $1$ & $\mathbf{e}_{1}$ & $\mathbf{e}_{2}$ & $\mathbf{e}_{3}$ & $\mathbf{e}_{4}$ & $\mathbf{e}_{5}$ & $\mathbf{e}_{6}$ & $\mathbf{e}_{7}$ \\
$\mathbf{e}_{1}$ & $\mathbf{e}_{1}$ & $-1$ & $\mathbf{e}_{4}$ & $\mathbf{e}_{5}$ & $-\mathbf{e}_{2}$ & $-\mathbf{e}_{3}$ & $-\mathbf{e}_{7}$ & $\mathbf{e}_{6}$ \\
$\mathbf{e}_{2}$ & $\mathbf{e}_{2}$ & $-\mathbf{e}_{4}$ & $-1$ & $\mathbf{e}_{6}$ & $\mathbf{e}_{1}$ & $\mathbf{e}_{7}$ & $-\mathbf{e}_{3}$ & $-\mathbf{e}_{5}$ \\
$\mathbf{e}_{3}$ & $\mathbf{e}_{3}$ & $-\mathbf{e}_{5}$ & $-\mathbf{e}_{6}$ & $-1$ & $-\mathbf{e}_{7}$ & $\mathbf{e}_{1}$ & $\mathbf{e}_{2}$ & $\mathbf{e}_{4}$ \\
$\mathbf{e}_{4}$ & $\mathbf{e}_{4}$ & $\mathbf{e}_{2}$ & $-\mathbf{e}_{1}$ & $\mathbf{e}_{7}$ & $-1$ & $-\mathbf{e}_{6}$ & $\mathbf{e}_{5}$ & $-\mathbf{e}_{3}$ \\
$\mathbf{e}_{5}$ & $\mathbf{e}_{5}$ & $\mathbf{e}_{3}$ & $-\mathbf{e}_{7}$ & $-\mathbf{e}_{1}$ & $\mathbf{e}_{6}$ & $-1$ & $-\mathbf{e}_{4}$ & $\mathbf{e}_{2}$ \\
$\mathbf{e}_{6}$ & $\mathbf{e}_{6}$ & $\mathbf{e}_{7}$ & $\mathbf{e}_{3}$ & $-\mathbf{e}_{2}$ & $-\mathbf{e}_{5}$ & $\mathbf{e}_{4}$ & $-1$ & $-\mathbf{e}_{1}$ \\
$\mathbf{e}_{7}$ & $\mathbf{e}_{7}$ & $-\mathbf{e}_{6}$ & $\mathbf{e}_{5}$ & $-\mathbf{e}_{4}$ & $\mathbf{e}_{3}$ & $-\mathbf{e}_{2}$ & $\mathbf{e}_{1}$ & $-1$ \\[2pt]
\hline
\end{tabular}
\end{center}
\end{table}\par
Here, we use the same labelling of the basis elements as used in \cite{Baez}. Additionally, we also need to consider the octonionic conjugate which is given by 
\begin{equation*}
\overline{x} = x_{0}\mathbf{e}_{0}-x_{1}\mathbf{e}_{1}-x_{2}\mathbf{e}_{2}-x_{3}\mathbf{e}_{3}-x_{4}\mathbf{e}_{4}-x_{5}\mathbf{e}_{5}-x_{6}\mathbf{e}_{6}-x_{7}\mathbf{e}_{7}.
\end{equation*}
The Euclidean norm from $\mathbb{R}^8$ then is reproduced by $x \cdot \overline{x}=\sum\limits_{j=0}^7 x_i^2 = |x|^2$, so every non-zero octonion $x$ is invertible via $x^{-1} = \dfrac{\overline{x}}{|x|^2}$.\par
While there are several possibilities to extend the classical function theory to octonions, we recall here the definition in the sense of the Riemann-approach, following the classical development of P. Dentoni and M. Sce \cite{DS}, K. Nono \cite{Nono}, the school of Xingmin-Li and Zhong Peng, see for instance \cite{XL2000} and others. In this sense we recall: 
\begin{definition}[Octonionic monogenicity]
Let $U \subseteq \mathbb{O}$ be an open set. A real-differentiable function $f:U \to \mathbb{O}$ is called left (right) octonionic monogenic if ${\cal{D}} f = 0$ (esp. $f {\cal{D}} = 0$). Here, ${\cal{D}}:= \dfrac{\partial }{\partial x_0} + \sum\limits_{i=1}^7 \mathbf{e}_i \dfrac{\partial }{\partial x_i}$ is the octonionic first order Cauchy-Riemann operator. If $f$ satisfies $\overline{{\cal{D}}}f = 0$ (resp. $f\overline{\cal{D}} = 0$), then we call $f$ left (right) octonionic anti-monogenic.
\end{definition}\par
Octonionic analysis has two crucial differences to classical Clifford analysis:
\begin{itemize}
\item[(i)] Octonionic analysis considers functions from $\mathbb{R}^8$ back into $\mathbb{R}^8$, while Clifford analysis addresses null-solutions to the Cauchy-Riemann operator defined on the paravector space $\mathbb{R}\oplus\mathbb{R}^7$ with values in the Clifford algebra $\mathcal{C}\ell_7$. $\mathcal{C}\ell_7$ however is a real vector space being isomorphic to $\mathbb{R}^{128}$.
\item[(ii)] Left(right) octonionic monogenic functions do neither form a right nor a left $\mathbb{O}$-module. See for example J. Kauhanen and H. Orelma in \cite{Kauhanen_3} for concrete examples. This fact significantly complicates the  development of a consistent function theory and the theory of generalised Hilbert function spaces in octonionic settings, see also \cite{CKS2023,ConKra2021,QR2021}.
\end{itemize}\par
Additionally, the lack of associativity leads to modifications of the classical integral formulae, such as for example the Stokes' formula \cite{XLT2008}:
\begin{equation}
\label{Stokes_formula_continuous}
\begin{array}{c}
\displaystyle \int\limits_{\partial G} g(x) \; (d\sigma(x)  f(x)) = \\
\displaystyle \int\limits_G \Bigg(   
g(x)({\cal{D}} f(x)) + (g(x){\cal{D}})f(x)  - \sum\limits_{j=0}^7 [e_j, {\cal{D}}g_j(x),f(x)]
\Bigg) dV,
\end{array}
\end{equation}
Here, the expression $[a,b,c] := (ab)c - a(bc)$ called the {\itshape associator} appears. This is an intrinsic feature, which is cancelled out in the cases of associativity. It is important to remark, that although the associator appears in most of octonionic constructions, it is nevertheless possible to introduce specific structures, where the associator would vanish. For example, it has been pointed out in \cite{XL2000}, that considering the two functions being octonionic monogenic and Stein-Weiss conjugate harmonics, i.e. $\dfrac{\partial g_j}{\partial x_i} = \dfrac{\partial g_i}{\partial x_j}$ for all $0\leq i<j\leq 7$, the associator will vanish.\par
Further, a generalisation of the Cauchy's integral formula to octonionic setting has been presented, see for instance  \cite{Nono,XL2002}:  
\begin{proposition}[Cauchy's integral formula]\label{cauchy1}
Let $U \subseteq \mathbb{O}$ be open and $G \subseteq U$ be an $8$-D compact oriented manifold with a strongly Lipschitz boundary $\partial G$. If $f\colon U \to \mathbb{O}$ is left octonionic monogenic, then for all $x \in G$
\begin{equation*}
f(x)= \frac{3}{\pi^4} \int\limits_{\partial G} q_{\bf 0}(y-x) \Big(d\sigma(y) f(y)\Big).
\end{equation*}
\end{proposition}
However, if the parenthesis were set differently, then one would obtain the different formula  
\begin{equation*}
\frac{3}{\pi^4} \int\limits_{\partial G} \Big( q_{\bf 0}(y-x) d\sigma(y)\Big) f(y)  =  f(x) +  \int\limits_G \sum\limits_{i=0}^7 \Big[q_{\bf 0}(y-x),{\cal{D}}f_i(y),\mathbf{e}_i  \Big] dy_0 \cdots dy_7, 
\end{equation*}
which involves the associator again.\par

\subsection{Discretisation of octonionic analysis}

Let us consider the unbounded uniform lattice $h\mathbb{Z}^{8}$ with the lattice constant $h>0$, which is defined in the classical way as follows
\begin{eqnarray*}
h \mathbb{Z}^{8} :=\left\{\mathbf{x} \in {\mathbb{R}}^{8}\,|\, \mathbf{x} = (m_{0}h, m_{1}h,\ldots, m_{7}h), m_{j} \in \mathbb{Z}, j=0,1,\ldots,7\right\}.
\end{eqnarray*}
Next, we define the classical forward and backward differences $\partial_{h}^{\pm j}$ as
\begin{equation}
\label{Finite_differences}
\begin{array}{lcl}
\partial_{h}^{+j}f(mh) & := & h^{-1}(f(mh+\mathbf{e}_jh)-f(mh)), \\
\partial_{h}^{-j}f(mh) & := & h^{-1}(f(mh)-f(mh-\mathbf{e}_jh)),
\end{array}
\end{equation}
for discrete functions $f(mh)$ with $mh\in h\mathbb{Z}^{8}$. In the sequel, we consider functions defined on $\Omega_{h} \subset  h\mathbb{Z}^{8}$ and taking values in octonions $\mathbb{O}$. As usual, all important properties such as the $l^{p}$-summability ($1\leq p<\infty$) are defined component-wisely.\par
A next step is to introduce discretisations of the Cauchy-Riemann operators in octonions. In contrast to the classical discrete Clifford analysis presented for instance in \cite{Brackx_1,FKS07}, where fundamental ideas of the Weyl calculus are used, we follow the alternative approach approach presented in \cite{Faustino_1} which uses a direct discretisation of the continuous Cauchy-Riemann (or Dirac) operators in terms of forward and backward finite difference operators. Following this approach, the non-associativity of octonionic multiplication can be respected, see \cite{Krausshar_1} for details. Hence, by using the finite difference operators~(\ref{Finite_differences}), we introduce a {\itshape discrete forward Cauchy-Riemann operator} $D^{+}\colon l^{p}(\Omega_{h},\mathbb{O})\to l^{p}(\Omega_{h},\mathbb{O})$ and a {\itshape discrete backward Cauchy-Riemann operators} $D^{-}\colon l^{p}(\Omega_{h},\mathbb{O})\to l^{p}(\Omega_{h},\mathbb{O})$ as follows
\begin{equation}
\label{Cauchy_Riemann_operators_discrete}
D^{+}_{h}:=\sum_{j=0}^{7} \mathbf{e}_j\partial_{h}^{+j}, \quad D^{-}_{h}:=\sum_{j=0}^{7} \mathbf{e}_j\partial_{h}^{-j}.
\end{equation}
Direct computations show that the star-Laplacian $\Delta_{h}$ is represented in this setting as follows:
\begin{equation*}
\Delta_{h} = \frac{1}{2}\left(D_{h}^{+}\overline{D_{h}^{-}}+D_{h}^{-}\overline{D_{h}^{+}}\right) \mbox{ with } \Delta_{h}:=\sum_{j=0}^{7}\partial_{h}^{+j}\partial_h^{-j},
\end{equation*}
where $\overline{D_{h}^{+}}$ and $\overline{D_{h}^{-}}$ are the discrete conjugated forward and backward Cauchy-Riemann operators, respectively:
\begin{equation*}
\overline{D_{h}^{-}}=\partial_{h}^{-0}-\sum_{j=1}^{7} \mathbf{e}_j\partial_{h}^{-j}, \quad \overline{D_{h}^{+}}=\partial_{h}^{+0}-\sum_{j=1}^{7} \mathbf{e}_j\partial_{h}^{+j}.
\end{equation*}\par
Since we consider discrete forward and backward Cauchy-Riemann operators, it is also necessary to distinguish between {\itshape discrete forward and backward monogenic functions}:
\begin{definition}\label{Definition:discrete_monogenic}
A function $f\in l^{p}(\Omega_{h},\mathbb{O})$ is called {\itshape discrete left forward monogenic} if $D_{h}^{+}f=0$ in $\Omega_{h}$. Respectively, a function $f\in l^{p}(\Omega_{h},\mathbb{O})$ is called {\itshape discrete left backward monogenic} if $D_{h}^{-}f=0$ in $\Omega_{h}$.
\end{definition}\par

\subsection{Discrete fundamental solution}

To construct discrete versions of the Borel-Pompeiu and Cauchy formulae, it is necessary to work with discrete fundamental solutions of the discrete Cauchy-Riemann operator and of the discrete Laplace operator. Thus, the following definition introduces a discrete fundamental solution to the discrete Cauchy-Riemann operators~(\ref{Cauchy_Riemann_operators_discrete}):
\begin{definition}
The function $E_{h}^{+}\colon h\mathbb{Z}^{8} \rightarrow \mathbb{O}$ is called a discrete fundamental solution of $D_{h}^{+}$ if it satisfies
\begin{equation*}
D_{h}^{+}E_{h}^{+} =\delta_h = \begin{cases} h^{-8}, & \mbox{for } mh=0,\\ 
0,& \mbox{for } mh\neq 0,
\end{cases}
\end{equation*} 
for all grid points $mh$ of $h\mathbb{Z}^{8}$. Analogously, the function $E_{h}^{-}\colon h\mathbb{Z}^{8} \rightarrow \mathbb{O}$ is called a discrete fundamental solution of $D_{h}^{-}$ if it satisfies
\begin{equation*}
D_{h}^{-}E_{h}^{-} =\delta_h = \begin{cases} h^{-8}, & \mbox{for } mh=0,\\ 
0,& \mbox{for } mh\neq 0,
\end{cases}
\end{equation*} 
for all grid points $mh$ of $h\mathbb{Z}^{8}$.
\end{definition}\par
As usual, a discrete fundamental solution can be constructed by means of the discrete Fourier transform of 
$u \in l^{p}\left( h\mathbb{Z}^8,\mathbb{O}\right)$, $1\leq p<+\infty$, 
\begin{equation*}
\mathbf{\xi} \mapsto \mathcal{F}_{h}u(\xi) = \sum_{m\in\mathbb{Z}^{8}}e^{i\langle mh,\xi \rangle}u(mh)h^8, \quad \xi\in\left[-\frac{\pi}{h},\frac{\pi}{h}\right]^{8}, 
\end{equation*}
where $\langle mh,\xi \rangle=h \sum\limits_{j=1}^{8}m_j \xi_j$. It is worth to underline, that the last expression implicitly introduces a complexified octonionic structure in the sense of $\mathbb{O}_{\mathbb{C}}:= \mathbb{O}\otimes_{\mathbb{R}}\mathbb{C}$, implying that the complex imaginary unit $i$ commutes with all real octonions.\par
The inverse transform is given by $\mathcal{F}_{h}^{-1}=\mathcal{R}_{h}\mathcal{F},$ where $\mathcal{F}$ is the classical continuous Fourier transform
\begin{equation*}
x \mapsto \mathcal{F}f(x) = \frac{1}{(2\pi)^{8}}\int_{\mathbb{R}^{8}} e^{-i \langle x,\xi \rangle} f(\xi)d\xi,
\end{equation*}
applied to an octonionic-valued function $f \in l^{p}\left( h\mathbb{Z}^8,\mathbb{O}\right)$ with $\mathrm{supp}\, f \in\left[-\frac{\pi}{h},\frac{\pi}{h}\right]^{8}$, and $\mathcal{R}_{h}$ denotes its restriction to the lattice $h\mathbb{Z}^{8}$. Note again that $f$ is $\mathbb{O}$-valued, so the expression $e^{-i \langle x,\xi \rangle} f(\xi)$ is formally embedded in $\mathbb{O}_{\mathbb{C}}$.\par 
Next, we recall the known symbols for the forward and backward differences $\partial_{h}^{\pm j}$, namely $\xi_{h}^{\pm j}=\mp h^{-1}\left( 1-e^{\mp ih\xi_j}\right)$, as well as the symbol for the star-Laplacian: 
\begin{equation*}
\mathcal{F}_h (\Delta_h u )(\xi)=d^2\mathcal{F}_h u(\xi) \mbox{ with } d^2=\frac{4}{h^2}\sum\limits_{j=0}^{7} \sin^2\left(\frac{\xi_j h}{2}\right).
\end{equation*}
Thus, applying the discrete Fourier transform to forward and backward Cauchy-Riemann operators leads to
\begin{equation*}
\mathcal{F}_h (D_{h}^{+} u)(\xi)=\left( \sum\limits_{j=0}^{7} \mathbf{e}_{j}\xi_{h}^{+j}\right) \mathcal{F}_h u(\xi), \quad \mathcal{F}_h (D_{h}^{-} u)(\xi)=\left( \sum\limits_{j=0}^{7} \mathbf{e}_{j}\xi_{h}^{-j}\right) \mathcal{F}_h u(\xi),
\end{equation*}
implying that the operators $D_{h}^{\pm}$ have complexified octonionic symbols $\widetilde{\xi}_{\pm} = \sum\limits_{j=0}^{7}\mathbf{e}_{j} \xi_{h}^{\pm j}$, respectively. Hence, the fundamental solutions $E_{h}^{\pm}$ can be expressed by
\begin{equation}
\label{Fundamental_solutions}
E_{h}^{\pm}=\mathcal{R}_{h}\mathcal{F} \left(\frac{\widetilde{\xi}_{\pm}}{d^2}\right)
=\sum\limits_{j=0}^{7}\mathbf{e}_{j}\mathcal{R}_{h} \mathcal{F} \left( \frac{\xi_{h}^{\pm j}}{d^2} \right).
\end{equation}
Evidently, the discrete fundamental solutions $E_{h}^{\pm}$ constructed above are different to the discrete fundamental solutions typically considered in the framework of discrete Clifford analysis, see for example \cite{CKKS}. The difference comes from the fact, that we work with a direct discretisation of the Cauchy-Riemann operators without involving the splitting of basis unit vectors, as it is done in discrete Clifford analysis.\par
Next, we will provide two basic properties of the discrete fundamental solutions $E_{h}^{\pm}$. For proving one of these properties, we need to recall before the following theorem \cite{Thomee}:
\begin{theorem}\label{Theorem:Thomee}
Let $n$ be the dimension of the Euclidean space and assume that $p_{1}$, $p_{2}$ are two positive integers with $p_{2}<p_{1}+n$. For a positive integer $N>0$ let $\kappa_{N}$ consider the set of functions of the form $T(\Theta)=\dfrac{T_{1}(\Theta)}{T_{2}(\Theta)}$, $0\neq \Theta \in Q_{\pi}$, where $T_{j}(\Theta)$ are trigonometric polynomials
\begin{equation*}
T_{j}(\Theta) = \sum\limits_{\mu} t_{j,\mu} e^{i \mu\cdot\Theta}, \quad j=1,2,
\end{equation*}
which satisfy the following conditions:
\begin{itemize}
\item[(i)] there are ordinary homogeneous polynomials $P_{j}(\Theta)$ of degree $p_{j}$, $j=1,2$, such that $T_{j}(\Theta)=P_{j}(\Theta)+o(|\Theta|^{p_{j}})$ when $\Theta \rightarrow 0$,
\item[(ii)] $|T_{2}(\Theta)|\geq N^{-1} |\Theta|^{p_{2}}$, $\Theta \in Q_{\pi}$,
\item[(iii)] $|t_{j,\mu}|\leq N$,
\item[(iv)] $t_{j,\mu}=0$ for $|\mu|>N$.
\end{itemize}
For any $N>0$ satisfying $(ii)$-$(iv)$ there is a constant $C$ such that for all $\mu$ (with integer components) and $T\in \kappa_{N}$,
\begin{equation}
\left| \int\limits_{Q_{\pi}} T(\Theta) e^{i \mu\cdot\Theta} d\Theta \right| \leq C(|\mu|+1)^{-(n+p_{1}-p_{2})}.
\end{equation}
\end{theorem}
These tools in hand now allow us to establish the following theorem:
\begin{theorem}
The discrete fundamental solutions $E_{h}^{\pm}$ to the discrete forward and backward Cauchy-Riemann operators satisfy:
\begin{itemize}
\item[(i)] $D_{h}^{\pm}E_{h}^{\pm}(mh) = \delta_{h}(mh)$, $m\in\mathbb{Z}^{8}$;
\item[(ii)] $E_{h}^{\pm}\in l^{p}(h\mathbb{Z}^{8},\mathbb{O})$ for $p>\frac{8}{7}$.
\end{itemize}
\end{theorem}

\begin{proof}
The proof of property (i) can be done by a straightforward calculations. To prove (ii), we are going to use the integral representation of the discrete fundamental solution $E_{h}^{+}$:
\begin{equation*}
E_{h}^{+}(mh)=\frac{1}{(2\pi)^{8}}\int\limits_{\xi\in[-\frac{\pi}{h},\frac{\pi}{h}]^{8}} \frac{\tilde{\xi}_{+}}{d^{2}}e^{-i\langle mh,\xi \rangle}d\xi, \quad m\in\mathbb{Z}^{8}.
\end{equation*}
Taking into account the definition of $\tilde{\xi}_{+}$ and that 
\begin{equation*}
\xi_{h}^{+j}=-\frac{1}{h}\left( 1-e^{-ih\xi_j}\right) = -\frac{1}{h}\left( 1-\cos(h\xi_j) + i\sin(h\xi_j)\right),
\end{equation*}
and after using known trigonometric identity, we get
\begin{equation*}
\begin{array}{rcl}
\displaystyle \left|E_{h}^{+}(mh)\right| & = & \displaystyle \left| \frac{1}{(2\pi)^{8}} \int\limits_{\xi\in[-\frac{\pi}{h},\frac{\pi}{h}]^{8}} \frac{\xi_{h}^{+j}}{d^{2}}e^{-i\langle mh,\xi \rangle}d\xi \right| \\
\\
& = & \displaystyle \left| \frac{1}{(2\pi)^{8}} \int\limits_{\xi\in[-\frac{\pi}{h},\frac{\pi}{h}]^{8}} \frac{2\sin^{2}\frac{h\xi_{j}}{2}}{hd^{2}}e^{-i\langle mh,\xi \rangle}d\xi \right. \\
\\
& & \displaystyle \left.+ i \frac{1}{(2\pi)^{8}} \int\limits_{\xi\in[-\frac{\pi}{h},\frac{\pi}{h}]^{8}} \frac{\sin(h\xi_j)}{hd^{2}}e^{-i\langle mh,\xi \rangle}d\xi \right|.
\end{array}
\end{equation*}
The conditons of Theorem~\ref{Theorem:Thomee} are satisfied for the two integrals above if $p_{1}=p_{2}=2$ for the first integral, and if $p_{1}=1$, $p_{2}=2$ for the second integral, see also \cite{Legatiuk} for a detailed discussion. Thus, the following estimate is obtained:
\begin{equation*}
\left| \frac{1}{(2\pi)^{8}} \int\limits_{\xi\in[-\frac{\pi}{h},\frac{\pi}{h}]^{8}} \frac{\xi_{h}^{+j}}{d^{2}}e^{-i\langle mh,\xi \rangle}d\xi \right| \leq \frac{C h}{\left(|mh|+h\right)^{8}} + \frac{C}{\left(|mh|+h\right)^{7}} \approx \mathcal{O}\left(\frac{1}{|mh|^{7}}\right).
\end{equation*}
Hence, it can be concluded that the fundamental solution $E_{h}^{+}$ belongs to $l^{p}(h\mathbb{Z}^{8},\mathbb{O})$ for $p>\frac{8}{7}$. The proof for the formula of $E_{h}^{-}$ can be performed analogously.
\end{proof}

\section{Discrete Stokes' and Borel-Pompeiu formulae}

We start this section by recalling the discrete octonionic Stokes' formulae from \cite{Krausshar_1}. Again, for shortening the notations, the long list of indices $m_{0},m_{1},\ldots,m_{7}$, will be omitted from the argument, i.e. we will simply write $f(mh)$ instead of $f(m_{0}h,m_{1}h,\ldots,m_{7}h)$.\par
We are particularly interested in studying upper and lower half-spaces (or half-lattices), which are defined as follows:
\begin{equation*}
\begin{array}{rcl}
h\mathbb{Z}_{+}^{8} & := & \left\{(h\underline{m},hm_{7})\colon \underline{m}\in\mathbb{Z}^{7},m_{7}\in\mathbb{Z}_{+}\right\}, \\
h\mathbb{Z}_{-}^{8} & := & \left\{(h\underline{m},hm_{7})\colon \underline{m}\in\mathbb{Z}^{7},m_{7}\in\mathbb{Z}_{-}\right\}.
\end{array}
\end{equation*}\par
The following two theorems express the discrete octonionic Stokes' formulae for the upper half-lattice and for the lower half-lattice, respectively \cite{Krausshar_1}:
\begin{theorem} 
The discrete Stokes' formula for the upper half-lattice $h\mathbb{Z}_{+}^{8}$ is given by
\begin{equation}
\label{DiscreteStokesFormula_upper_half_lattice}
\begin{array}{c}
\displaystyle \sum_{m\in \mathbb{Z}_{+}^{8}}  \left\{ \left[ g(mh)D_h^{+}\right] f(mh) - g(mh) \left[ D_h^{-}f(mh) \right]  \right\} h^8 \\
\displaystyle = \sum\limits_{\underline{m}\in \mathbb{Z}^{7}} \mathbf{e}_{7}\left(g(\underline{m},1)f_{k}(\underline{m},0)\right) h^{8}
\end{array}  
\end{equation}
for all discrete functions $f$ and $g$ such that the series converge.
\end{theorem}
\begin{theorem} 
The discrete Stokes' formula for the lower half-lattice $h\mathbb{Z}_{-}^{8}$ is given by
\begin{equation}
\label{DiscreteStokesFormula_lower_half_lattice}
\begin{array}{c}
\displaystyle \sum_{m\in \mathbb{Z}_{-}^{8}}  \left\{ \left[ g(mh)D_h^{+}\right] f(mh) - g(mh) \left[ D_h^{-}f(mh) \right]  \right\} h^8 \\
\displaystyle = -\sum\limits_{\underline{m}\in \mathbb{Z}^{7}} \mathbf{e}_{7}\left(g(\underline{m},0)f_{k}(\underline{m},-1)\right) h^{8}
\end{array}  
\end{equation}
for all discrete functions $f$ and $g$ such that the series converge.
\end{theorem}\par
\begin{remark}
We would like to remark that the discrete Stokes' formulae introduced above do not contain the associator, which appears in the continuous case as an intrinsic term. The non-appearance of the associator therefore represents a surprising effect of the discrete setting. Moreover, the cancellation of the associator is not related to the consideration of the forward or backward discrete Cauchy-Riemann operators, either. Note that if we consider instead for example the central discrete Cauchy-Riemann operator:
\begin{equation*}
\tilde{D}_{h} := \frac{1}{2}\left(D_{h}^{+}+D_{h}^{-}\right),
\end{equation*}
then the discrete octonionic formula for the whole space would have the form
\begin{equation*}
\begin{array}{rl}
& \displaystyle \sum_{m\in \mathbb{Z}^{8}}  \left\{ \left[ g(mh)\tilde{D}_{h}\right] f(mh) - g(mh) \left[ \tilde{D}_{h}f(mh) \right]  \right\} h^8 = 0.
\end{array}  
\end{equation*}
For the sake of completeness, let us briefly outline the proof of this formula. We start with the first summand ($h$ is omitted for the sake of abbreviation)
\begin{equation*}
\begin{array}{rl}
& \displaystyle \sum\limits_{m\in \mathbb{Z}^{8}} \left[ g(m)\tilde{D}_{h}\right]f(m)h^{8} = \sum\limits_{m\in \mathbb{Z}^{8}} \sum\limits_{j=0}^{7} \frac{1}{2}\left[\partial_{h}^{+j}g(m)\mathbf{e}_{j}+\partial_{h}^{-j}g(m)\mathbf{e}_{j}\right] f(m) h^{8} \\
\\
= & \displaystyle \sum\limits_{m\in \mathbb{Z}^{8}} \sum\limits_{j=0}^{7} \frac{1}{2} \sum \limits_{i=0}^{7}\sum \limits_{k=0}^{7} \left[\partial_{h}^{+j}g_{i}(m)\mathbf{e}_{i}\mathbf{e}_{j}+\partial_{h}^{-j}g_{i}(m)\mathbf{e}_{i}\mathbf{e}_{j}\right] f_{k}(m)\mathbf{e}_{k} h^{8}.
\end{array}
\end{equation*}
By using the relation $(\mathbf{e}_{i}\mathbf{e}_{j})\mathbf{e}_{k}=-\mathbf{e}_{i}(\mathbf{e}_{j}\mathbf{e}_{k})$ and the definitions of finite differences, we get
\begin{equation*}
\begin{array}{rl}
& \displaystyle \sum\limits_{m\in \mathbb{Z}^{8}} \sum\limits_{j=0}^{7}\frac{1}{2} \sum \limits_{i=0}^{7}\sum \limits_{k=0}^{7} \left[-\left(g_{i}(m+\mathbf{e}_{j})-g_{i}(m)\right)f_{k}(m)\mathbf{e}_{i}(\mathbf{e}_{j}\mathbf{e}_{k})-\right. \\
& \left. \displaystyle - \left(g_{i}(m)-g_{i}(m-\mathbf{e}_{j})\right)f_{k}(m)\mathbf{e}_{i}(\mathbf{e}_{j}\mathbf{e}_{k})\right]h^{8} \\
\\
= & \displaystyle \sum\limits_{m\in \mathbb{Z}^{8}} \sum\limits_{j=0}^{7}\frac{1}{2} \sum \limits_{i=0}^{7}\sum \limits_{k=0}^{7} \left[-g_{i}(m+\mathbf{e}_{j})\mathbf{e}_{i}f_{k}(m)+g_{i}(m-\mathbf{e}_{j})\mathbf{e}_{i}f_{k}(m)\right](\mathbf{e}_{j}\mathbf{e}_{k})h^{8}.
\end{array}
\end{equation*}
Performing a change of variables in the latter expression, we get
\begin{equation*}
\begin{array}{rl}
& \displaystyle \sum\limits_{m\in \mathbb{Z}^{8}} \sum\limits_{j=0}^{7}\frac{1}{2} \sum\limits_{i=0}^{7}\sum \limits_{k=0}^{7} \left[-g_{i}(m)\mathbf{e}_{i}f_{k}(m-\mathbf{e}_{j})+g_{i}(m)\mathbf{e}_{i}f_{k}(m+\mathbf{e}_{j})\right](\mathbf{e}_{j}\mathbf{e}_{k})h^{8} \\
\\
= & \displaystyle \sum\limits_{m\in \mathbb{Z}^{8}} \sum\limits_{j=0}^{7} \frac{1}{2}\sum\limits_{i=0}^{7}\sum \limits_{k=0}^{7} \left[g_{i}(m)\mathbf{e}_{i}\left(f_{k}(m+\mathbf{e}_{j})+f_{k}(m-\mathbf{e}_{j})\right)\right](\mathbf{e}_{j}\mathbf{e}_{k})h^{8} \\
\\
= & \displaystyle \sum\limits_{m\in \mathbb{Z}^{8}} \sum\limits_{j=0}^{7} \frac{1}{2}\sum\limits_{i=0}^{7}\sum \limits_{k=0}^{7} g_{i}(m)\mathbf{e}_{i}\left(\partial_{h}^{+j}\mathbf{e}_{j} f_{k}\mathbf{e}_{k}+\partial_{h}^{-j}\mathbf{e}_{j} f_{k}\mathbf{e}_{k}\right)h^{8} \\
\\
= & \displaystyle \sum\limits_{m\in \mathbb{Z}^{8}} g(m)\left[\tilde{D}_{h}f(m)\right]h^{8}.
\end{array}
\end{equation*}
We just have proved the discrete octonionic Stokes' formula for the whole space in the case of the central discrete Cauchy-Riemann operator. This highlights that the constructions in the discrete settings exhibit an essentially different nature than the constructions in the continuous case, where the associstor appears as an intrinsic ingredient. 
\end{remark}\par
By means of the discrete discrete octonionic Stokes' formule~(\ref{DiscreteStokesFormula_upper_half_lattice})-(\ref{DiscreteStokesFormula_lower_half_lattice}), the discrete octononionic Borel-Pompeiu formulae can be introduced:
\begin{theorem}
Let $E_{h}^{+}$ be the discrete fundamental solution to the discrete Cauchy-Riemann operator $D_{h}^{+}$. Then the discrete octonionic Borel-Pompeiu formula for the upper half-lattice $h\mathbb{Z}_{+}^{8}$ is given by
\begin{equation}
\label{Discrete_Borel_Pompeiu_upper_half_lattice}
\begin{array}{c}
\displaystyle \sum_{n\in \mathbb{Z}_{+}^{8}}  E_{h}^{+}(nh-mh) \left[ D_h^{-}f(mh) \right] h^8 \\
\displaystyle + \sum\limits_{\underline{n}\in \mathbb{Z}^{7}} \mathbf{e}_{7}\left(E_{h}^{+}(\underline{n}h-\underline{m}h,1)f_{k}(\underline{m},0)\right) h^{8} = \left\{\begin{array}{cl}
0, & m\notin\mathbb{Z}_{+}^{8}, \\
f(mh), & m\in\mathbb{Z}_{+}^{8},
\end{array} \right.
\end{array}  
\end{equation}
for any discrete function $f$ such that the series converge.
\end{theorem}
\begin{proof}
To prove the discrete octonionic Borel-Pompeiu formula for the upper half-lattice, we use at first the discrete octonionic Stokes's formula for the upper half-lattice~(\ref{DiscreteStokesFormula_upper_half_lattice}), and replace therein the function $g$ by the shifted discrete fundamental solution $E_{h}^{+}(\cdot-mh)$ with $m\in\mathbb{Z}_{+}^{8}$. Next, by taking into account the definition of the discrete fundamental solution, we note that $\left[ E_{h}^{+}(nh-mh)D_{h}^{+}\right]=0$ for $n \neq m$ and $\left[ E_{h}^{+}(nh-mh)D_{h}^{+}\right]=h^{-8}$ for $n=m$. Thus, the formula is established.
\end{proof}\par
By using the discrete octonionic Borel-Pompeiu formula~(\ref{Discrete_Borel_Pompeiu_upper_half_lattice}) and requiring that the function $f$ is discrete left backward monogenic in $hZ_{+}^{8}$, we immediately arrive at the discrete octonionic Cauchy formula:
\begin{theorem}\label{Theorem:Discrete_Cauchy_formula_upper}
Let $f$ be a discrete left backward monogenic function with respect to the operator $D_{h}^{-}$, and let $E_{h}^{+}$ be the discrete fundamental solution to the operator $D_{h}^{+}$. Then the discrete octonionic Cauchy formula for the upper half-lattice $hZ_{+}^{8}$ is given by
\begin{equation}
\label{Discrete_Cauchy_formula_upper_half_lattice}
\sum\limits_{\underline{n}\in \mathbb{Z}^{7}} \mathbf{e}_{7}\left(E_{h}^{+}(\underline{n}h-\underline{m}h,1)f_{k}(\underline{m},0)\right) h^{8} = \left\{\begin{array}{cl}
0, & m\notin\mathbb{Z}_{+}^{8}, \\
f(mh), & m\in\mathbb{Z}_{+}^{8},
\end{array} \right.
\end{equation}
which holds for any discrete function $f$ such that the series converge.
\end{theorem}
Analogously we may introduce discrete octonionic Borel-Pompeiu and Cauchy formulae for the lower half-lattice $h\mathbb{Z}_{-}^{+}$:
\begin{corollary}\label{Corollary:Borel_Pompeiu_Cauchy_lower}
Let $f$ be a discrete left backward monogenic function with respect to the operator $D_{h}^{-}$, and let $E_{h}^{+}$ be the discrete fundamental solution to operator $D_{h}^{+}$. Then the discrete octonionic Borel-Pompeiu formula for the lower half-lattice $h\mathbb{Z}_{-}^{8}$ is given by
\begin{equation}
\label{Discrete_Borel_Pompeiu_lower_half_lattice}
\begin{array}{c}
\displaystyle \sum_{n\in \mathbb{Z}_{-}^{8}}  E_{h}^{+}(nh-mh) \left[ D_h^{-}f(mh) \right] h^8 \\
\displaystyle - \sum\limits_{\underline{n}\in \mathbb{Z}^{7}} \mathbf{e}_{7}\left(E_{h}^{+}(\underline{n}h-\underline{m}h,0)f_{k}(\underline{m},-1)\right) h^{8} = \left\{\begin{array}{cl}
0, & m\notin\mathbb{Z}_{-}^{8}, \\
f(mh), & m\in\mathbb{Z}_{-}^{8},
\end{array} \right.
\end{array}  
\end{equation}
for any discrete function $f$ such that the series converge. In the case when $f$ is a discrete left backward monogenic function with respect to the operator $D_{h}^{-}$, we obtain the discrete octonionic Cauchy formula for the lower half-lattice $h\mathbb{Z}_{-}^{8}$ in the form 
\begin{equation}
\label{Discrete_Cauchy_formula_lower_half_lattice}
-\sum\limits_{\underline{n}\in \mathbb{Z}^{7}} \mathbf{e}_{7}\left(E_{h}^{+}(\underline{n}h-\underline{m}h,0)f_{k}(\underline{m},-1)\right) h^{8} = \left\{\begin{array}{cl}
0, & m\notin\mathbb{Z}_{-}^{8}, \\
f(mh), & m\in\mathbb{Z}_{-}^{8}.
\end{array} \right.
\end{equation}
\end{corollary}\par
Theorem~\ref{Theorem:Discrete_Cauchy_formula_upper} and Corollary~\ref{Corollary:Borel_Pompeiu_Cauchy_lower} immediately lead to the definition of discrete octonionic Cauchy transforms for the upper and lower half-lattices:
\begin{definition}
For a discrete $l^p$-function $f$, $1\leq p<+\infty$, defined on the boundary layers $(\underline{m},0)$ and $(\underline{m},1)$ the discrete octonionic Cauchy transform for the upper half-lattice $h\mathbb{Z}_{+}^{8}$ is defined by
\begin{equation}\label{Cauchy_transforms_upper}
\mathcal{C}_{\mathbb{O}}^{+}[f](mh) := \sum\limits_{\underline{n}\in \mathbb{Z}^{7}} \mathbf{e}_{7}\left(E_{h}^{+}(\underline{n}h-\underline{m}h,1)f_{k}(\underline{m},0)\right) h^{8}.
\end{equation}
Analogously, for a discrete $l^p$-function $f$, $1\leq p<+\infty$, defined on the boundary layers $(\underline{m},0)$ and $(\underline{m},-1)$ the discrete octonionic Cauchy transform for the lower half-lattice $h\mathbb{Z}_{-}^{8}$ is defined by
\begin{equation}\label{Cauchy_transforms_lower}
\mathcal{C}_{\mathbb{O}}^{-}[f](mh) := -\sum\limits_{\underline{n}\in \mathbb{Z}^{7}} \mathbf{e}_{7}\left(E_{h}^{+}(\underline{n}h-\underline{m}h,0)f_{k}(\underline{m},-1)\right) h^{8}.
\end{equation}
\end{definition}
It is worth to mention, that alike in the continuous case also in discrete Clifford analysis, the discrete octonionic Cauchy formulae and Cauchy transform indicate the dependence of (discrete) monogenic functions on their boundary values. However, as a consequence of the lack of associativity, the signs in the discrete octonionic Borel-Pompeiu and Cauchy formulae~(\ref{Discrete_Borel_Pompeiu_upper_half_lattice})-(\ref{Discrete_Cauchy_formula_lower_half_lattice}) differ from the results in discrete Clifford analysis, compare for example with \cite{CKKS}. Additionally, in line with the discussion in \cite{Krausshar_1}, the discretisation of the continuous octonionic analysis has the interesting effect that the associator disappears from the constructions.\par
Let us now present some properties of the discrete octonionic Cauchy transforms:
\begin{theorem}
Let us consider the discrete upper half-lattice $h\mathbb{Z}_{+}^{8}$ and the lower half-lattice $h\mathbb{Z}_{-}^{8}$.  Then the discrete Cauchy transforms~(\ref{Cauchy_transforms_upper})-(\ref{Cauchy_transforms_lower}) satisfy the following properties:
\begin{itemize}
\item[(i)] The interior and exterior Cauchy transforms have the following mapping properties:
\begin{equation*}
\begin{array}{ll}
\displaystyle \mathcal{C}_{\mathbb{O}}^{+} \colon l^{p}(h\mathbb{Z}_{+}^{8},\mathbb{O}) \to l^{q}(h\mathbb{Z}_{+}^{8},\mathbb{O}), & 1 \leq p,q \leq \infty, \\
\displaystyle \mathcal{C}_{\mathbb{O}}^{-} \colon l^{p}(h\mathbb{Z}_{-}^{8},\mathbb{O}) \to l^{q}(h\mathbb{Z}_{-}^{8},\mathbb{O}), & 1 \leq p < \infty, \frac{8}{7} < q < \infty. \\
\end{array}
\end{equation*}
\item[(ii)] $D_{h}^{+} \mathcal{C}_{\mathbb{O}}^{+}[f](mh)=0$, $\forall\, m=(\underline{m},m_{7})$ with $m_{7}\geq 1$.
\item[(iii)] $D_{h}^{+} \mathcal{C}_{\mathbb{O}}^{-}[f](mh)=0$, $\forall\, m=(\underline{m},m_{7})$ with $m_{7}\leq -1$.
\end{itemize}
\end{theorem}
\begin{proof}
The proof of this theorem can be performed along the same ideas as in \cite{CKKS,Cerejeiras}, and, therefore, we will only mention that the first statement follows from a direct application of H\"older's inequality and the properties of the discrete fundamental solution $E_{h}^{+}$. The proof of (ii) and (iii) is done by straightforward application of the discrete forward Cauchy-Riemann operator $D_{h}^{+}$ and applying its properties.
\end{proof}\par

\section{Discrete octonionic Hardy spaces}

The aim of this section is to construct discrete octonionic Hardy spaces. Looking at the results in discrete Clifford analysis related to Hardy spaces, see for example \cite{CKKS,Cerejeiras_2,Cerejeiras}, it becomes evident that a typical approach to introduce discrete Hardy spaces will be to work with Fourier transforms on boundary layers of the discrete fundamental solution. In this way, discrete Riesz kernels can be defined, and, hence, discrete Plemelj (or Hardy) projections are then introduced. Alternatively, one could follow the classical continuous approach for defining Hardy spaces and adapt it to the discrete setting. For the purpose of a better understanding the effect of non-associativity, both approaches will be presented in this section.\par
At first, we recall from \cite{CKS2023} the following definition of an octonionic Hilbert space, which straightforwardly extends to the discrete setting:
\begin{definition}\label{Definition:Hilbert_space}
An octonionic Hilbert space $H$ is a left $\mathbb{O}$-module with an octonion-valued inner product $\langle \cdot,\cdot\rangle\colon H\times H\to\mathbb{O}$ such that $\left(H,\langle \cdot,\cdot\rangle_{0}\right)$ is a real Hilbert space, where $\langle \cdot,\cdot\rangle_{0}:= \mathrm{Re}\langle \cdot,\cdot\rangle$. The octonion-valued inner product is supposed to satisfy for all $f,g,h\in H$ and for all $\alpha\in\mathbb{O}$ the following rules:
\begin{itemize}
\item[(i)] Additivity: $\langle f+g,h\rangle = \langle f,h \rangle + \langle g,h \rangle$;
\item[(ii)] Hermitian property: $\langle g,f \rangle = \overline{\langle f,g \rangle}$;
\item[(iii)] Strict positivity: $\langle f,f \rangle \in \mathbb{R}^{\geq 0}$ and $\langle f,f \rangle = 0$ iff $f=0$;
\item[(iv)] $\mathbb{R}$-homogeneity: $\langle fr,g \rangle = \langle f,g \rangle r$ for all $r\in\mathbb{R}$;
\item[(v)] $\mathbb{O}$-homogeneity: $\langle f\alpha,f \rangle = \langle f,f \rangle\alpha$;
\item[(vi)] $\mathbb{O}$-para-linearity: $\langle f\alpha,g\rangle_{0} = \mathrm{Re}\left(\langle f\alpha,g\rangle\right) = \mathrm{Re}\left(\langle f,g\rangle\alpha\right)$.
\end{itemize}
\end{definition}\par
Before introducing the notion of a discrete octonionic Hardy space, let us fix the notation $\gamma_{h}$ for the discrete boundary of $\Omega_{h}$ (0-layer in the case of half-lattices). Now we introduce the following notion of a discrete octonionic Hardy space, which is adapted from the definition presented in \cite{CKS2023}:
\begin{definition}\label{Definition:discrete_Hardy_space}
The discrete octonionic Hardy space $\mathcal{H}_{h}^{2}(\Omega_{h},\mathbb{O})$ is the closure of the set of $l^{2}(\gamma_{h},\mathbb{O})$ functions that are discrete left monogenic inside of $\Omega_{h}$ and extendable to the boundary $\gamma_{h}$.
\end{definition}
It is important to remark, that we have not distinguished in this definition between discrete forward and backward monogenic functions as it has been introduced in Definition~\ref{Definition:discrete_monogenic}. The reason for this is to avoid too many specific definitions, which do not generally change the notion of a discrete octonionic Hardy space, since both types of discrete monogenic functions can be used. Further, Definition~\ref{Definition:discrete_Hardy_space} requires a discrete monogenic function being extendable to the boundary $\gamma_{h}$, implying that an extension operator must be defined. Although there are several possibilities to define such an operator, arguably the most straightforward way is to work with the Fourier transform of the discrete fundamental solution $E_{h}^{+}$, as it is done in the discrete Clifford analysis, see again \cite{CKKS,Cerejeiras_2,Cerejeiras}. Hence, we will introduce an extension operator based on this idea, and thus, establish a clear connection between the discrete Clifford analysis approach and the \textquotedblleft continuous theory-based\textquotedblright\, approach.\par
Next, we need to introduce an octonion-valued inner product on the discrete boundary $\gamma_{h}$:
\begin{definition}
For discrete octonionic function $f,g\in l^{2}(\gamma_{h})$, we introduce the following $\mathbb{O}$-valued inner products:
\begin{equation}
\label{Inner_product_upper_half}
\langle f,g\rangle_{\gamma_{h}}^{h\mathbb{Z}_{+}^{8}} := \sum_{mh\in \gamma_{h}} \overline{\left(-\mathbf{e}_{7}g(mh)\right)}\left(-\mathbf{e}_{7}f(mh)\right)h^{2},
\end{equation}
if the upper half-lattice is considered, and
\begin{equation}
\label{Inner_product_lower_half}
\langle f,g\rangle_{\gamma_{h}}^{h\mathbb{Z}_{-}^{8}} := \sum_{mh\in \gamma_{h}} \overline{\left(\mathbf{e}_{7}g(mh)\right)}\left(\mathbf{e}_{7}f(mh)\right)h^{2},
\end{equation}
if the lower half-lattice is considered.
\end{definition}
It is then easy to verify the following statement:
\begin{proposition}
The sets $\left(\mathcal{H}_{h}^{2}(\gamma_{h},\mathbb{O}),\langle \cdot,\cdot\rangle_{\gamma_{h}}^{h\mathbb{Z}_{+}^{8}}\right)$ and $\left(\mathcal{H}_{h}^{2}(\gamma_{h},\mathbb{O}),\langle \cdot,\cdot\rangle_{\gamma_{h}}^{h\mathbb{Z}_{-}^{8}}\right)$ are octonionic Hilbert spaces in the sense of Definition~\ref{Definition:Hilbert_space}.
\end{proposition}\par
The next step is to establish the theory of discrete Hardy spaces in the sense of discrete Clifford analysis. This requires a definition of discrete Riesz kernels (convolution kernels) implying that the behaviour of the discrete fundamental solution $E_{h}^{+}$ on boundary layers needs to be studied. Let us recall the integral representation of the discrete fundamental solution
\begin{equation*}
E_{h}^{+}(mh)=\frac{1}{(2\pi)^{8}}\int\limits_{\xi\in[-\frac{\pi}{h},\frac{\pi}{h}]^{8}} \frac{\tilde{\xi}_{+}}{d^{2}}e^{-i\langle mh,\xi \rangle}d\xi, \quad m\in\mathbb{Z}^{8}
\end{equation*}
with $\widetilde{\xi}_{+} = \sum\limits_{j=0}^{7}\mathbf{e}_{j} \xi_{h}^{+j}$. To study Fourier symbols on the boundary layers, i.e. for $m_{7}\in\left\{-1,0,1\right\}$, we apply the $7$--dimensional discrete Fourier transform to the discrete fundamental solution:
\begin{equation*}
\begin{array}{c}
\displaystyle \mathcal{F}_h^{(7)} E_{h}^{+}(\underline \eta, m_{7}h) = \sum_{\underline mh \in h\mathbb{Z}^{7}} e^{-i h \langle \underline{m}, \underline \eta \rangle }  \left[ \frac{1}{(2\pi)^8} \int\limits_{\left[-\frac{\pi}{h},\frac{\pi}{h}\right]^{7} } e^{-i h \langle m, \xi \rangle } \frac{\widetilde{\xi}_{+}}{d^2} d\xi \right] \\
\displaystyle =  \frac{1}{(2\pi)^{7}}    \int\limits_{\left[-\frac{\pi}{h}, \frac{\pi}{h} \right]^{7}}     \sum_{\underline mh \in h\mathbb{Z}^{7}} e^{-i h \langle \underline m, \underline \eta - \underline \xi \rangle } \underbrace{\left[\frac{1}{2\pi}  \int\limits_{-\frac{\pi}{h}}^{\frac{\pi}{h}} e^{-ih m_{7}\xi_{7} } \frac{\widetilde{\xi}_{+}}{d^2} d\xi_{7}\right]}_{(I)}  d\underline \xi.
\end{array}
\end{equation*}
Let us study now the integral $(I)$:
\begin{equation*}
\begin{array}{lcl}
(I) & = & \displaystyle \frac{1}{2\pi}  \int\limits_{-\frac{\pi}{h}}^{\frac{\pi}{h}} e^{-ih m_{7}\xi_{7} } \frac{\widetilde{\xi}_{+}}{d^2} d\xi_{7} = \frac{1}{2\pi}  \int\limits_{-\frac{\pi}{h}}^{\frac{\pi}{h}} e^{-ih m_{7}\xi_{7} } \frac{\widetilde{\underline{\xi}}_{+}+\widetilde{\xi}_{+,7}}{\underline{d}^2 + \frac{4}{h^{2}}\sin^{2}\left(\frac{\xi_{7}h}{2}\right) } d\xi_{7} \\
 & = & \displaystyle \frac{1}{2\pi}  \int\limits_{-\frac{\pi}{h}}^{\frac{\pi}{h}} e^{-ih m_{7}\xi_{7} } \frac{\widetilde{\underline{\xi}}_{+}-\mathbf{e}_{7}\frac{1}{h}\left(1-e^{-ih\xi_{7}}\right)}{\underline{d}^2 + \frac{4}{h^{2}}\sin^{2}\left(\frac{\xi_{7}h}{2}\right) } d\xi_{7} \\
 & = & \displaystyle \frac{\widetilde{\underline{\xi}}_{+}}{2\pi}  \int\limits_{-\frac{\pi}{h}}^{\frac{\pi}{h}}  \frac{e^{-ih m_{7}\xi_{7} }}{\underline{d}^2 + \frac{4}{h^{2}}\sin^{2}\left(\frac{\xi_{7}h}{2}\right) } d\xi_{7} - \frac{\mathbf{e}_{7}}{2\pi h}  \int\limits_{-\frac{\pi}{h}}^{\frac{\pi}{h}} e^{-ih m_{7}\xi_{7} } \frac{1-e^{-ih\xi_{7}}}{\underline{d}^2 + \frac{4}{h^{2}}\sin^{2}\left(\frac{\xi_{7}h}{2}\right) } d\xi_{7} \\
 & = & \displaystyle \frac{\widetilde{\underline{\xi}}_{+}}{2\pi}  \int\limits_{-\frac{\pi}{h}}^{\frac{\pi}{h}}  \frac{e^{-ih m_{7}\xi_{7} }}{\underline{d}^2 + \frac{4}{h^{2}}\sin^{2}\left(\frac{\xi_{7}h}{2}\right) } d\xi_{7} - \frac{\mathbf{e}_{7}}{2\pi h}  \int\limits_{-\frac{\pi}{h}}^{\frac{\pi}{h}}  \frac{e^{-ih m_{7}\xi_{7} }}{\underline{d}^2 + \frac{4}{h^{2}}\sin^{2}\left(\frac{\xi_{7}h}{2}\right) } d\xi_{7} \\
& & \displaystyle + \frac{\mathbf{e}_{7}}{2\pi h}  \int\limits_{-\frac{\pi}{h}}^{\frac{\pi}{h}} \frac{e^{-ih\xi_{7}(m_{7}+1)}}{\underline{d}^2 + \frac{4}{h^{2}}\sin^{2}\left(\frac{\xi_{7}h}{2}\right) } d\xi_{7}.
\end{array}
\end{equation*}
The integrals above need to be calculated for $m_{7}\in\left\{-1,0,1\right\}$, and these calculations have been presented in \cite{CKKS} for the case of splitting of basis unit elements. Hence, the Fourier symbols of the discrete fundamental solutions on the layers $m_{7}=-1$, $m_{7}=0$, and $m_{7}=1$ are given by
\begin{equation*}
\begin{array}{rcl}
\displaystyle \mathcal{F}_h^{(7)} E_{h}^{+}(\underline \xi, 0) & = & \displaystyle \frac{\widetilde{\underline{\xi}}_{+}}{\underline{d}\sqrt{4+h^{2}\underline{d}^{2}}} - \mathbf{e}_{7}\left(\frac{1}{2}-\frac{h\underline{d}}{2\sqrt{4+h^{2}\underline{d}^{2}}}\right), \\
\displaystyle \mathcal{F}_h^{(7)} E_{h}^{+}(\underline \xi, h) & = & \displaystyle \frac{\widetilde{\underline{\xi}}_{+}}{\underline{d}}\left(\frac{2+h^{2}\underline{d}^{2}}{2\sqrt{4+h^{2}\underline{d}^{2}}}-\frac{h\underline{d}}{2}\right) + \mathbf{e}_{7}\left(-\frac{3h\underline{d}+h^{3}\underline{d}^{3}}{2\sqrt{4+h^{2}\underline{d}^{2}}}+\frac{h^{2}\underline{d}^{2}+1}{2}\right), \\
\displaystyle \mathcal{F}_h^{(7)} E_{h}^{+}(\underline \xi, -h) & = & \displaystyle \frac{\widetilde{\underline{\xi}}_{+}}{\underline{d}}\left(\frac{2+h^{2}\underline{d}^{2}}{2\sqrt{4+h^{2}\underline{d}^{2}}}-\frac{h\underline{d}}{2}\right) - \mathbf{e}_{7}\left(-\frac{3h\underline{d}+h^{3}\underline{d}^{3}}{2\sqrt{4+h^{2}\underline{d}^{2}}}+\frac{h^{2}\underline{d}^{2}+1}{2}\right). \\
\end{array}
\end{equation*}\par
Next, following ideas in \cite{CKKS}, we introduce the pair of operators:
\begin{equation*}
\begin{array}{rcl}
H_{+}f & := & \displaystyle \mathcal{F}_{h}^{-1}\left[\mathbf{e}_{7}\frac{\widetilde{\underline{\xi}}_{+}}{\underline{d}}\frac{2}{h\underline{d}-\sqrt{4+h^{2}\underline{d}^{2}}}\right]\mathcal{F}_{h}f, \\
H_{-}f & := & \displaystyle -\mathcal{F}_{h}^{-1}\left[\mathbf{e}_{7}\frac{\widetilde{\underline{\xi}}_{+}}{\underline{d}}\frac{h\underline{d}-\sqrt{4+h^{2}\underline{d}^{2}}}{2}\right]\mathcal{F}_{h}f,
\end{array}
\end{equation*}
which fulfil the condition $(H_+)^2 = (H_- )^2=I$. By help of these operators, we can formulate conditions for a function to be a discrete boundary value of a discrete octonionic monogenic function in $h\mathbb{Z}_{+}^{8}$ or $h\mathbb{Z}_{-}^{8}$:
\begin{equation*}
\begin{array}{lcl}
f(mh) & = & H_+f(mh), \mbox{ for } m_{7}=1, \\
f(mh) & = & H_-f(mh), \mbox{ for } m_{7}=-1.
\end{array}
\end{equation*}
This enables us now to introduce another definition of discrete octonionic Hardy spaces:
\begin{definition}
The space of discrete functions $f\in l^{p}(h\mathbb{Z}_{+}^{8},\mathbb{O})$ whose discrete 2D-Fourier transform fulfills $f=  H_+f$ for $m_{7}=1$ is called the {\itshape upper discrete octonionic Hardy space} and it is denoted by $h_{p,h\mathbb{Z}_{+}^{8}}^{+}$. Analogously, the space of discrete functions $f\in l^{p}(h\mathbb{Z}_{-}^{8},\mathbb{O})$ whose discrete 2D-Fourier transform fulfills $f=  H_-f$ for $m_{7}=-1$ is called the {\itshape lower discrete octonionic Hardy space} and it is denoted by $h_{p,h\mathbb{Z}_{-}^{8}}^{-}$.
\end{definition}\par
By means of the operators $H_{+}$ and $H_{-}$, the discrete Plemelj or Hardy projections can be now introduced as follows
\begin{equation*}
P_+=\frac{1}{2}\left(I+ H_+ \right) \mbox{ and } P_-=\frac{1}{2}\left(I+H_-\right).
\end{equation*}
Additionally, combining these projections with the previous definition we get
\begin{equation*}
f\in h_{p,h\mathbb{Z}_{+}^{8}}^{+} \Longleftrightarrow P_{+}f=f, \mbox{ and } f\in h_{p,h\mathbb{Z}_{-}^{8}}^{-} \Longleftrightarrow P_{-}f=f.
\end{equation*}\par
Finally, we define two extension operators, which extend a discrete function from layers $m_{7}=-1$ and $m_{7}=1$ to the boundary layer $m_{7}=0$, see \cite{CKK15} for the details:
\begin{definition}
The upper extension operator, denoted as $\mathcal{A}_{+}$, is an operator extending a function given on the boundary layer $m_{7}=1$ to the boundary layer $m_{7}=0$, i.e. it is a mapping $\mathcal{A}_{+} \colon l^{p}(h\mathbb{Z}^{7}) \to l^{p}(h\mathbb{Z}^{7})$ given by
\begin{equation*}
\mathcal{A}_{+} := \mathcal{F}_h^{(n-1)} \left[ \frac{ \widetilde{\underline{\xi}}_{+}}{\underline{d}} \left( \dfrac{2}{\sqrt{4+h^{2} \underline{d}^{2}} - h \underline{d}} \right) \right].
\end{equation*}
Similarly, the lower extension operator, denoted by $\mathcal{A}_{-}$, is an operator extending a function given on the boundary layer $m_{7}=-1$ to the boundary layer $m_{7}=0$, i.e. it is a mapping $\mathcal{A}_{-}\colon l^{p}(h\mathbb{Z}^{7}) \to l^{p}(h\mathbb{Z}^{7})$ given by
\begin{equation*}
\mathcal{A}_{-} := \mathcal{F}_h^{(n-1)} \left[ \left( \frac{\sqrt{4+h^{2}\underline{d}^{2}} + h\underline{d}}{\sqrt{4+h^{2} \underline{d}^{2}} - h \underline{d}} \right) \right].
\end{equation*}
\end{definition}
This definition rounds off the discussion around discrete octonionic Hardy spaces, which we have initiated in Definition~\ref{Definition:discrete_Hardy_space}. Now it is clear how an extension of a discrete function to the boundary layer $m_{7}=0$ can be explained.\par

\section{Summary}

In this paper, we have continued the development of discrete octonionic analysis by introducing discrete Borel-Pompeiu and Cauchy formulae, and defined discrete octonionic Hardy spaces for half-spaces. Moreover, we have discussed two approaches for constructing discrete Hardy spaces: (i) by a direct discretisation of the continuous case; and (ii) by using an approach used in discrete Clifford analysis by studying Fourier symbols of the discrete fundamental solution of the discrete Cauchy-Riemann operator on the boundary layers. Both approaches complement each other and contribute to a better understanding of the discrete octonionic setting.\par
A very surprising result is that the associator, which appears in the continuous case, does not appear in the discrete setting. For a better understanding of this effect, the discrete Stokes' formula for the whole space has been proved  also for the central discrete Cauchy-Riemann operator. Nevertheless, the associator does not appear even in that case, which underlines the particularity of the constructions in the discrete setting.\par
The results presented in this paper provide us with a powerful basic toolkit for a further development of discrete octonionic analysis. In particular, the discrete versions of octonionic Hardy spaces that we introduced allow us to study concrete boundary value problems for monogenic functions in the discrete octonionic setting. This also opens the door to study Calderon-Zygmund type operators in this context. Furthermore, the discrete toolkit allow us us develop numerical algorithms for solving some particular physical problems arising in the unification of particle physics and gravity as illustrated in a series of recent works in this direction, see again \cite{Burdik,NSM}, such as already mentioned roughly in the introductory text.   

Further, after understanding the difficulties arising on the way discretising octonionic analysis for the half-space settings, the results of this paper can subsequentially be extended to the case of considering bounded domains in $\mathbb{R}^{8}$.\par

\medskip\par

{\bf Data Availability Statement}. Data sharing not applicable to this article as no datasets were generated or analysed during the current study.



%
%


\begin{thebibliography}{99}

\bibitem{Baez}
Baez, J.: The octonions. Bulletin of the American Mathematical Society, 39, 145-205, 2002. 

\bibitem{BDS} 
Brackx, F., Delanghe, R., Sommen, F.: Clifford analysis. Pitman Research Notes in Mathematics, 76, Boston, 1982.

\bibitem{Brackx_1}
Brackx, F., De Schepper, H., Sommen, F., Van de Voorde, L.: Discrete Clifford analysis: a germ of function theory. Hypercomplex Analysis, Birkh\"auser Basel, 37-53, 2009.

\bibitem{Burdik} 
Burdik, C., Catto, S., G\"urcan, Y., Khalfan, A., Kurt, L., Kato La, V.: $SO(9,1)$ group and examples of analytic functions. Journal of Physics: Conference Series, 1194, 012016, 2019.

\bibitem{Cerejeiras_1} 
Cerejeiras, P., Faustino, N., Vieira, N.: Numerical Clifford analysis for nonlinear Schr\"odinger problem. Numerical Methods for Partial Differential Equations, 24(4), 1181-1202, 2008.

\bibitem{CKKS}
Cerejeiras, P., K\"ahler, U., Ku, M., Sommen, F.: Discrete Hardy spaces. Journal of Fourier Analysis and Applications, 20(4), 715-750, 2014.

\bibitem{CKK15}
Cerejeiras, P., K\"ahler, U., Ku, M.: Discrete Hilbert boundary value problems on half lattices. Journal of Difference Equations and Applications, 21(12), 1277-1304, 2015.

\bibitem{Cerejeiras_2}
Cerejeiras, P., K\"ahler, U., Legatiuk, A., Legatiuk, D.: Boundary values of discrete monogenic functions over bounded domains in $\mathbb{R}^{3}$. In: D. Alpay, M. Vajiac (eds), Linear Systems, Signal Processing and Hypercomplex Analysis. Operator Theory: Advances and Applications, Vol. 275. Birkh\"auser, 149-165, 2020.

\bibitem{Cerejeiras}
Cerejeiras, P., K\"ahler, U., Legatiuk, A., Legatiuk, D.: Discrete Hardy spaces for bounded domains in $\mathbb{R}^{n}$. Complex Analysis and Operator Theory, 15(4), 2021.

\bibitem{CKS2023} 
Colombo, F., Krau{\ss}har, R.S., Sabadini, I.: Octonionic monogenic and slice monogenic Hardy and Bergman spaces, submitted for publication, 2023. Preprint: https://www.mate.polimi.it/collezioni-digitali-di-dipartimento/?cdd1=Submit\&lg=en

\bibitem{ConKra2021}
Constales, D., Krau\ss har, R.S.: Octonionic Kerzman-Stein operators. Complex Analysis and Operator Theory, 15:104, 2021.

\bibitem{DS}
Dentoni, P., Sce, M.: Funzioni regolari nell’algebra di Cayley. Rend. Sem. Mat. Univ. Padova, 50, 251-267, 1973.

\bibitem{Faustino_1} 
Faustino, N., K\"ahler, U.: Fischer decomposition for difference Dirac operators. Advances in Applied Clifford Algebras, 17, 37–58, 2007.

\bibitem{FKS07} 
Faustino, N., K\"ahler, U., Sommen, F.: Discrete Dirac operators in Clifford analysis. Advances in Applied Clifford Algebras, 17(3), 451-467, 2007.

\bibitem{FL} 
Frenod, E., Ludkowski, S.V.: Integral operator approach over octonions to solution of nonlinear PDE. Far East Journal of Mathematical Sciences, 2017.

\bibitem{G} 
Gogberashvili, M.: Octonionic geometry and conformal transformations. International Journal of Geometric Methods in Modern Physics, 13(7), 1650092, 2016.

\bibitem{Guerlebeck_1}
G\"urlebeck, K., Hommel, A.: On finite difference Dirac operators and their fundamental solutions. Advances in Applied Clifford Algebras, 11, 89-106, 2001.

\bibitem{GKR} 
G\"urlebeck, K.,  K\"ahler, U., Ryan, J., Spr\"o{\ss}ig, W.: Clifford analysis over unbounded domains. Advances in applied mathematics, 19(2), 216-239, 1997.

\bibitem{GS2} 
G\"urlebeck, K., Spr\"o{\ss}ig, W.: Quaternionic and Clifford calculus for physicists and engineers. John Wiley \& Sons, Chichester-New York, 1997.

\bibitem{GTBook} 
G\"ursey, F., Tze, H.: On the role of division and Jordan algebras in particle physics. World Scientific, Singapore, 1996.

\bibitem{Kauhanen_2}
Kauhanen, J., Orelma, H.: Cauchy-Riemann operators in octonionic analysis. Advances in the Applied Clifford Algebras, 28:1, 2018.

\bibitem{Kauhanen_3}
Kauhanen, J., Orelma, H.: On the structure of octonion regular functions. Advances in the Applied Clifford Algebras, 29:77, 2019.

\bibitem{Krausshar_1}
Krau\ss har, S., Legatiuk, A., Legatiuk, D.: Towards discrete octonionic analysis. In: V. Vasilyev (Ed.), Differential Equations, Mathematical Modeling and Computational Algorithms, Springer, accepted for publication, 2022.

\bibitem{Legatiuk}
Legatiuk, A.: Discrete potential and function theories on a rectangular lattice and their applications. PhD dissertation, Weimar, 2022.

\bibitem{McIntosh} 
McIntosh, A.: Clifford algebras, Fourier theory, singular integrals, and harmonic functions on Lipschitz domains. In: J. Ryan (Ed.), Clifford Algebras in Analysis and Related Topics, CRC Press, Boca Raton, 33–87, 1996.

\bibitem{NSM} 
Najarbashi, G., Seifi, B., Mirzaei, S.: Two- and three-qubit geometry, quaternionic and octonionic conformal maps, and intertwining stereographic projection. Quantum Inf Process, 15, 509-528, 2016.

\bibitem{Nono} 
N$\hat{o}$no, K.: On the octonionic linearization of Laplacian and octonionic function theory. Bull. Fukuoka Univ. Ed. Part III, 37, 1-15, 1988.

\bibitem{QR2021} 
Huo, Q., Ren G.: Para-linearity as the nonassociative counterpart of linearity. The Journal of Geometric Analysis, 32:304, 2022.

\bibitem{Thomee}
Thom\'{e}e, V.: Discrete interior Schauder estimates for elliptic difference operators. SIAM Journal of Numerical Analysis, 5, 626-645, 1968.

\bibitem{WL2018} 
Wang J., Li, X.: The octonionic Bergman kernel for the unit ball. Advances in the Applied Clifford Algebras, 28:60, 2018.

\bibitem{WL2020} 
Wang J., Li, X.: The octonionic Bergman kernel for the half space. Advances in the Applied Clifford Algebras, 30:57, 2020.
 
\bibitem{XL2000} 
Li, X.-M., Peng, L.-Z.: On Stein-Weiss conjugate harmonic function and octonion analytic function.  Approximation Theory and its Applications, 16, 28-36, 2000.

\bibitem{XL2002} 
Li, X.-M., Peng, L.-Z.: The Cauchy integral formulas on the octonions. Bulleting of Belgian Mathematical Society, 9, 47-62, 2002. 

\bibitem{XLT2008} 
Li, X.-M., Peng, L.-Z., Qian, T.: Cauchy integrals on Lipschitz surfaces in octonionic space. Journal of Mathematical Analysis and Applications, 343, 763–777, 2008.

\end{thebibliography}


\end{document}